\documentclass[11pt]{amsart}

\usepackage{mathrsfs}
\usepackage{amsmath,amssymb,amsthm}
\usepackage{graphicx}
\usepackage{latexsym}
\usepackage{mathrsfs}
\usepackage{tikz}
\usepackage{comment}
\usetikzlibrary{cd,arrows,matrix,backgrounds,positioning,calc,decorations.markings,decorations.pathmorphing,decorations.pathreplacing}
\tikzset{black/.style={circle,fill=black,inner sep=3pt,outer sep=3pt},
         white/.style={circle,fill=white,draw=black,inner sep=3pt,outer sep=3pt},
}
\input xy

\xyoption{all}
\newtheorem{theorem}{Theorem}[section]

\newtheorem{thm}[theorem]{Theorem}
\newtheorem{cor}[theorem]{Corollary}
\newtheorem{lemm}[theorem]{Lemma}
\newtheorem{prop}[theorem]{Proposition}

\theoremstyle{definition}
\newtheorem{definition-theorem}[theorem]{Definition-Theorem}
\newtheorem{defi}[theorem]{Definition}

\newtheorem{exam}[theorem]{Example}

\newcommand{\tors}{\mathsf{tors}}
\newcommand{\torf}{\mathsf{torf}}
\newcommand{\ftors}{\mathsf{f\mbox{-}tors}}
\newcommand{\ftorf}{\mathsf{f\mbox{-}torf}}

\newcommand{\Hasse}{\mathsf{Hasse}}
\newcommand{\shard}{\mathsf{shard}}
\newcommand{\shardint}{\mathsf{shard\mbox{-}int}}
\newcommand{\jirr}{\mathsf{j\mbox{-}irr}}
\newcommand{\simp}{\mathsf{sim}}
\newcommand{\brick}{\mathsf{brick}}
\newcommand{\sbrick}{\mathsf{sbrick}}

\newcommand{\fLwide}{\mathsf{f_L\mbox{-}wide}}
\newcommand{\fLsbrick}{\mathsf{f_L\mbox{-}sbrick}}
\newcommand{\fRsbrick}{\mathsf{f_R\mbox{-}sbrick}}

\newcommand{\twosilt}{\mathsf{2\mbox{-}silt}}
\newcommand{\twopsilt}{\mathsf{2\mbox{-}psilt}}
\newcommand{\sw}{\mathsf{s.w}\hspace{.01in}}
\newcommand{\Up}{\mbox{\rm Up}\hspace{.01in}}
\newcommand{\Lo}{\mbox{\rm Lo}\hspace{.01in}}
\newcommand{\covered}{{\,\,<\!\!\!\!\cdot\,\,\,}}
\newcommand{\covers}{{\,\,\,\cdot\!\!\!\! >\,\,}}

\newcommand{\Db}{\mathsf{D}^{\rm b}}
\newcommand{\Kb}{\mathsf{K}^{\rm b}}

\newcommand{\TT}{\operatorname{\mathcal T}\nolimits}
\newcommand{\FF}{\operatorname{\mathcal F}\nolimits}
\newcommand{\VV}{\operatorname{\mathcal V}\nolimits}
\newcommand{\T}{\operatorname{\mathsf{T}}\nolimits}
\newcommand{\Cc}{\operatorname{\mathsf{C}}\nolimits}
\newcommand{\U}{\operatorname{\mathsf{U}}\nolimits}
\newcommand{\F}{\operatorname{\mathsf{F}}\nolimits}
\newcommand{\MM}{\operatorname{\mathcal M}\nolimits}
\newcommand{\NN}{\operatorname{\mathcal N}\nolimits}
\newcommand{\LL}{\operatorname{\mathcal L}\nolimits}
\newcommand{\PP}{\operatorname{\mathcal P}\nolimits}
\newcommand{\QQ}{\operatorname{\mathcal Q}\nolimits}
\newcommand{\UU}{\operatorname{\mathcal U}\nolimits}
\newcommand{\RR}{\operatorname{\mathcal R}\nolimits}
\newcommand{\JJ}{\operatorname{\mathcal J}\nolimits}

\newcommand{\C}{\mathfrak{C}}
\newcommand{\CC}{\mathscr{C}}

\newcommand{\R}{\mathbb{R}}

\newcommand{\sS}{\mathcal{S}}

\renewcommand{\span}{\operatorname{span}\nolimits}

\newcommand{\Filt}{\mathsf{Filt}}

\newcommand{\thick}{\mathsf{thick}} 
\newcommand{\silt}{\mathsf{silt}}
\newcommand{\wide}{\mathsf{wide}}
\newcommand{\swide}{\mathsf{single\mbox{-}wide}}

\newcommand{\SSS}{\mathbb{S}}

\newcommand{\add}{\mathsf{add}\nolimits}
\newcommand{\proj}{\mathsf{proj}}

\newcommand{\Hom}{\operatorname{Hom}\nolimits}

\newcommand{\RHom}{\mathbf{R}\strut\kern-.2em\operatorname{Hom}\nolimits}

\def\dim{\mathop{\mathrm{dim}}\nolimits}

\def\Im{\mathop{\mathrm{Im}}\nolimits}

\def\Hom{\mathop{\mathrm{Hom}}\nolimits}
\def\End{\mathop{\mathrm{End}}\nolimits}

\def\RHom{\mathop{\mathbb R\mathrm{Hom}}\nolimits}
\DeclareMathOperator{\moduleCategory}{\mathsf{mod}} \renewcommand{\mod}{\moduleCategory}

\newcommand{\W}{\mathsf{W}\hspace{.01in}}
\newcommand{\Sub}{\mathsf{Sub}\hspace{.01in}}
\newcommand{\Fac}{\mathsf{Fac}\hspace{.01in}}
\def\add{{\mathsf{add}}}

\begin{document}
\title[Shard theory for $g$-fans]{Shard theory for $g$-fans}

\author{Yuya Mizuno}
\address{Faculty of Liberal Arts, Sciences and Global Education, Osaka Metropolitan University, 1-1 Gakuen-cho, Naka-ku, Sakai, Osaka 599-8531, Japan}
\email{yuya.mizuno@omu.ac.jp}
\thanks{The author is supported by 
Grant-in-Aid for Scientific Research 20K03539.}
\begin{abstract}
For a finite dimensional algebra $A$, the notion of $g$-fan $\Sigma(A)$ is defined from two-term silting complexes of $\Kb(\proj A)$ in the real Grothendieck group  $K_0(\proj A)_{\mathbb{R}}$. 
In this paper, we discuss the theory of shards to  $\Sigma(A)$, which was originally defined for a hyperplane arrangement. 
We establish a correspondence between the set of join-irreducible elements of torsion classes of $\mod A$ and the set of shards of $\Sigma(A)$ for $g$-finite algebra $A$. 
Moreover, we show that the semistable region of a brick of $\mod A$ is  exactly given by a shard. 
We also give a poset isomorphism of shard intersections and wide subcategories of $\mod A$.

\end{abstract}
\maketitle
\tableofcontents

\section{Introduction}

Let $A$ be a finite dimensional algebra. 
One of the fundamental aims of the representation theory of algebras is to 
study important objects and subcategories of the module category $\mod A$ or its derived category $\Db(\mod A)$. 
For this aim, many interesting notions and theories have been developed such as $\tau$-tilting modules, semibricks, silting complexes, simple-minded collections and so on (see \cite{AIR,Ri,MS,A1,AI,KV,KY} for example).
Recently, it has been found that combinatorial viewpoint plays an important role to investigate these notions. 
Among others, we can define a fan $\Sigma(A)$, called \emph{$g$-fan} (Definition \ref{g-fan}), which is defined in the real Grothendieck group $K_0(\proj A)_{\mathbb{R}}$ by the notion of $g$-vectors of two-term silting complexes of $\Kb(\proj A)$.
Surprisingly, the $g$-fan reflects not only the property of the partial order and mutation of the two-term silting complexes, but also uncovers several nontrivial results, which are difficult to know in a purely representation-theoretic way as shown in \cite{AHIKM,AMN}.

In this paper, we study further property of $g$-fans and, in particular, we explain that we can understand bricks, wide subcategories and canonical join representations in terms of $g$-fans. 
Our results can be summarized as follows.

\begin{thm}\label{}
Let $A$ be a finite dimensional algebra which is $g$-finite  (Definition \ref{g-finite}) and $\Sigma(A)$ the $g$-fan of $A$.
From $\Sigma(A)$, we can determine the following objects.
\begin{itemize}
\item Join irreducible elements of the set $\tors A$ of torsion classes.
\item A canonical join  representation of elements of the set $\tors A$ of torsion classes.  
\item The set $\brick A$ of bricks.
\item The set $\sbrick A$ of semibricks. 
\item The set $\swide A$ of wide subcategories generated by a single object.
\item The set $\wide A$ of wide subcategories. 
\item The partial order of $\wide A$.
\end{itemize}
\end{thm}

Thus the $g$-fan contains several fundamental data of the representation theory.   
To establish the theorem, we use the theory of \emph{shards}. 
Shards are certain codimension 1 cones and they were originally introduced by Reading for a hyperplane arrangement (see \cite{Re1,Re3,Re4}).  
Shards allow a new construction of noncrossing partition lattices and they were also used to provide a geometric description of lattice congruences \cite{Re4}. 
We extend the notion of shards to our setting (since $g$-fan is not necessary a hyperplane arrangement) and show that shards are quite natural and important objects to study bricks and wide subcategories. 
The key ingredient of our results is the following relationship between 
bricks and shards through King's stability conditions \cite{Ki}.

\begin{thm}[Theorem \ref{main1}]\label{1}
A shard gives precisely the semistable region of a brick. 
Moreover, we have a bijection between the set $\brick A$ of 
bricks  and the set $\shard A$ of shards of $\Sigma(A)$ by 
$$\xymatrix{\Theta(-):\brick A\ar[r]^{}&  \shard A,
}$$ 
where 
$\Theta(X):=\{\theta\in K_0(\proj A)_\mathbb{R}\mid X:\textnormal{$\theta$-semistable}\}$. 
\end{thm}

The main result of \cite{DIRRT} gives a bijection between forcing equivalence classes of arrows of the Hasse quiver of torsion classes of $\mod A$ and isomorphism classes of bricks. 
From this viewpoint, the above result shows  that shards give a geometric counterpart of 
forcing equivalence classes.  
Moreover, this result can be regarded as a generalization of \cite{T1}.
Indeed, let $\Pi$ be a preprojective algebra of Dynkin type. Then the $g$-fan $\Sigma(\Pi)$ is the Coxeter fan \cite[Theorem 8.4]{AHIKM} and $\Sigma(\Pi)$ is a hyperplane arrangement and shards in our definition coincides with the original definition. 
In particular, Theorem \ref{1} implies \cite[Theorem 6]{T1}, where  it was shown that a shard in the Coxeter arrangement gives precisely the semistable region of a brick of $\mod\Pi$. 

Moreover, we can give an explicit relationship between join-irreducible elements, wide subcategory, bricks and shards together with several known results \cite{A1,BCZ,DIRRT,MS,Ri}.

\begin{thm}[see Theorems \ref{bij jirr and shard},\ref{main1} for details]\label{2}
We have the following commutative diagrams, and all maps are bijections.  
\[
\begin{xy}
(-45,0)*{\jirr (\C(A))}="A",
(45,0)*{\shard A}="B",
(-45,-25)*{\brick A}="C",
(45,-25)*{\swide A}="D",
(-45,0)*{}="E",
(45,0)*{}="F",
(-45,-25)*{}="G",
(45,-25)*{}="H",

\ar@<1ex> "A";"B"^{\Sigma(-)}
\ar@<1ex> "B";"A"^{\JJ(-)}
\ar@<1ex> "A";"C"^{B(-)}
\ar@<1ex> "C";"A"^{\T(-)}
\ar@<1ex> "B";"D"^{\W(-)}
\ar@<1ex> "D";"B"^{\Theta(-)}
\ar@<1ex> "C";"D"^{\Filt(-)}
\ar@<1ex> "D";"C"^{\simp(-)}
\end{xy}
\]
\end{thm}

Thus shards are also quite fundamental objects from the viewpoint of the representation theory.
Moreover, we study intersections of shards (Definition \ref{defi shard int}). 
It was shown in \cite{Re3} that the partial order of shard intersections admit several nice properties such as being graded.  
From the viewpoint of the representation theory, this partial order corresponds to the one of wide subcategories, and they also provide a geometric description of canonical join representations in terms of $g$-fans.

\begin{thm}[see Theorems \ref{bij jirr and shard}, \ref{bij jirr and shard2} for details]\label{main2}
\begin{itemize}
\item[(1)] For any $\RR\in\C(A)$, we have 
$$\RR=\bigvee_{i=1}^k\JJ(\Sigma_i),$$ 
where $\Sigma_i$ runs over all lower shards of $\RR$ and it is a canonical join representation.

\item[(2)] 
There exists a poset anti-isomorphism
$$\xymatrix{\wide A\ar@<0.5ex>[rr]^(.45){\Theta(-)}&
&\ar@<0.6ex>[ll]^(.55){\W(-)} \shardint A.
}$$

\item[(3)]
We have the following commutative diagrams, and all maps  are 
bijections.   
\[
\begin{xy}
(-45,0)*{\C (A)}="A",
(45,0)*{\shardint A}="B",
(-45,-25)*{\sbrick A}="C",
(45,-25)*{\wide A}="D",
(-45,0)*{}="E",
(45,0)*{}="F",
(-45,-25)*{}="G",
(45,-25)*{}="H",

\ar@<1ex> "A";"B"^{\bigcap\Sigma(-)}
\ar@<1ex> "B";"A"^{\bigvee\JJ(-)}
\ar@<1ex> "A";"C"^{\{B(-)\}}
\ar@<1ex> "C";"A"^{\vee \TT(-)}
\ar@<1ex> "B";"D"^{\W(-)}
\ar@<1ex> "D";"B"^{\Theta(-)}
\ar@<1ex> "C";"D"^{\Filt(-)}
\ar@<1ex> "D";"C"^{ \simp(-)}
\end{xy}
\]
\end{itemize}
\end{thm}

Theorem \ref{main2} provides a direct connection between shard intersections and wide subcategories.  
In \cite{En}, he explained how to obtain the partial order of wide subcategory from the one of torsion classes. 
Theorem \ref{bij jirr and shard2} can be regarded as a geometric 
counterpart of this relationship. 
As an application, we can also show that the partial order of $\wide A$ is graded.

\section{Preliminaries}
In this section, we give some background materials and recall necessary results.

\subsection{Finite lattice}
We recall basic terminology of a partially ordered set.

\begin{defi}
Let $P$ be a partially ordered set.
\begin{itemize}
\item[(1)]
For $x,y,z \in P$,
the element $z$ is called the \emph{meet} of $x$ and $y$
if $z$ is the maximum element satisfying $z \le x$ and $z \le y$.
In this case, $z$ is denoted by $x \wedge y$.
\item[(2)]
For $x,y,z \in P$,
the element $z$ is called the \emph{join} of $x$ and $y$
if $z$ is the minimum element satisfying $z \ge x$ and $z \ge y$.
In this case, $z$ is denoted by $x \vee y$.
\item[(3)]
The set $P$ is called a \textit{lattice}
if $P$ admits the meet $x \wedge y$ and the join $x \vee y$ for any $x,y \in P$.
\end{itemize}
\end{defi}
In this paper, we basically consider  finite lattices.

\begin{defi}\label{def join}
Let $P$ be a finite lattice.
\begin{itemize}
\item[(1)] An element of $x\in P$ is called \emph{join-irreducible} 
if it is not the minimum element of $P$ and if $x=y\vee z$ 
for some $y,z \in P$, then $y = x$ or $z = x$, or equivalently, the join-irreducible elements are those which cover precisely the one element.
We denote by $\jirr(P)$ the set of join-irreducible elements in $P$.

\item[(2)] 
We call $C\subset P$ a \emph{canonical join representation} if
\begin{itemize}
\item[(i)] $x=\bigvee_{c \in C} c$. 
\item[(ii)] For any proper subset $C' \subsetneq C$,
the join $\bigvee_{c \in C'} c$ never coincides with $x$. 
\item[(iii)] If $U \subset P$ satisfies the properties (i) and (ii), then, 
for every $c \in C$, there exists $u \in U$ such that $ c\le u$.
\end{itemize}

In this case, we also call $x=\bigvee_{c \in C} c$ a canonical join representation. Note that if $x=\bigvee_{c \in C} c$ is a canonical join representation, then it is unique and each element $c\in C$ is join-irreducible. 

\item[(3)] $P$ is called semidistributive if it satisfies the two following conditions:
\begin{itemize}
\item[(i)] For $a,b,c\in P$, $a\wedge b=a\wedge c$ implies $a\wedge b=a\wedge(b\vee c)$.
\item[(ii)] For $a,b,c\in P$, $a\vee b=a\vee c$ implies $a\vee b=a\vee(b\wedge c)$.
\end{itemize}
\end{itemize}
\end{defi}

Let $P$ be a finite poset. 
For $x,y \in P$, if $x$ covers $y$, that is, 
$x>y$ and there does not exist $z\in P$ satisfying $x>z>y$, then we denote by $x\covers y$. 
Recall that the \emph{Hasse quiver} $\Hasse (P)$ of a poset $P$ has the set $P$ of vertices, and an arrow $x\to y$ if $x\covers y$.


\subsection{Torsion classes, wide subcategories, semibricks and 2-term silting complexes}


Throughout the paper, 
let $A$ be a finite dimensional algebra over a field $k$.
We denote by $\mod A$ the category of finitely generated right $A$-modules and by $\proj A$ the category of finitely generated projective $A$-modules.
We denote by $\Db(\mod A)$ the bounded derived category of $\mod A$ and by $\Kb(\proj A)$ the bounded homotopy category of $\proj A$. 
For an object $X$, we denote by $|X|$ the number of non-isomorphic indecomposable direct summands of $X$. 
For a full subcategory $\Cc \subset \mod A$, 
$\Filt \Cc$ consists of the objects $M$ 
such that there exists a sequence $0 =M_0 \subset M_1 \subset \cdots \subset M_n=M$
with $M_i/M_{i-1} \in \add \Cc$.

We recall the notion of torsion classes.

\begin{defi}\label{tosion class}
\begin{itemize}
\item[(1)] A full subcategory of $\mod A$ is called a \emph{torsion class} (respectively, \emph{torsion-free class}) if it is closed under extensions and factor modules (respectively, submodules).
A torsion class $\T$ (resp. torsion-free class $\F$) is called \emph{functorially finite} if there exists $M\in\mod A$ satisfying $\T=\Fac M$ (respectively, $\F=\Sub M$).
\item[(2)] A full subcategory $\W$ of $\mod A$ is called a \emph{wide subcategory} if it is closed under kernels, cokernels, and extensions of $\mod A$. 

\item[(3)] A wide subcategory $\W$ is called  \emph{single} if $\W$ has only one simple object (i.e., $\W=\Filt(S)$, where $S$ is the simple of $\W$). Moreover, 
a wide subcategory $\W$ is called  \emph{functorially finite} if 
the smallest torsion class containing $\W$ is functorially finite.

\end{itemize}
\end{defi}

We denote by $\tors A$ (respectively, $\ftors A$, $\torf A$, $\ftorf A$, $\wide A$, $\swide A$) the set of torsion classes (respectively, functorially finite torsion classes, torsion-free classes, functorially finite torsion-free classes, wide subcategories, single wide subcategories) in $\mod A$. These categories are naturally regarded as posets defined by inclusion.

Next we recall some objects which give the above categories.

\begin{defi}\label{semibricks}
\begin{itemize}
\item[(1)] 
A module $S$ in $\mod A$ is called a \textit{brick} 
if $\End_A(S)$ is a division $K$-algebra.
We denote by $\brick A$ the set of isoclasses of bricks in $\mod A$.  
\item[(2)] 
A subset $\sS \subset \brick A$ is called a \emph{semibrick} 
if $\Hom_A(S_i,S_j)=0$ holds for any $S_i \ne S_j \in \sS$.
We denote by $\sbrick A$ the set of semibricks in $\mod A$.   
\item[(3)] We say that a semibrick $\sS$ is \emph{left finite} (resp. \emph{right finite}) if the smallest torsion class (resp. torsion-free class)  $\T(\sS)\subset\mod A$ (resp. $\F(\sS)\subset\mod A$) containing $\sS$ is functorially finite. 
We denote by $\fLsbrick A$ (resp. $\fRsbrick A$) the set of left finite (resp. right finite) semibricks in $\mod A$.   
\end{itemize}
\end{defi}

\begin{defi}\label{silting}
Let $T=(T^i,d^i)\in\Kb(\proj A)$.
\begin{enumerate}
\item $T$ is called  \emph{presilting} if $\Hom_{\Kb(\proj A)}(T,T[\ell])=0$ for all  $\ell>0$.
\item $T$ is called \emph{silting} if it is presilting and $\Kb(\proj A)=\thick T$, where $\thick(T)$ denotes the smallest thick subcategory containing $T$. 
\item $T$ is called  \emph{2-term} if $T^i = 0$ for all $i\not= 0,-1$. 
\end{enumerate}
\end{defi}
We denote by $\twosilt A$ (resp. $\twopsilt A$)
the set of isomorphism classes of basic 
2-term silting (resp. 2-term presilting) complexes of $\Kb(\proj A)$.
For an integer $i$ with $1\leq i\leq n$, 
we define 
$$\twopsilt^{i} (A):=\{T\in \twopsilt (A) \mid |T|=i \}.$$ 
Recall that a 2-term presilting complex $T$ is silting if and only if $|T|=|A|$ holds and hence $\twopsilt^n (A) =\twosilt (A)$  \cite[Proposition 3.3]{AIR}. 
Moreover, for any $U\in\twopsilt A$, there exists $T\in\twosilt A$ such that $U\in\add T$ (\cite[Theorem 2.10]{AIR}).

Moreover, for $T,U\in\twosilt A$, we write $T\ge U$ if $\Hom_{\Kb(\proj A)}(T,U[\ell])=0$ holds for all positive integers $\ell$. Then $(\silt A,\ge)$ and hence $(\twosilt A,\ge)$ is a partially ordered set \cite{AI}.

Then we recall an important connection between torsion classes, 2-term silting complexes and semibricks as follows.

\begin{theorem}\cite{AIR,A1}
\label{poset iso}
We have a  bijection
$$\T(-):\fLsbrick A\longrightarrow\ftors A.$$
Moreover, we have a poset isomorphism
$$\Fac(H^0(-)):\twosilt A\longrightarrow\ftors A.$$
\end{theorem}

\begin{theorem}\cite{Ri,A1}
\label{bij sbrick wide}
We have bijections 
$$\xymatrix@C20pt@R10pt{\sbrick A\ar@<0.5ex>[rr]^{\Filt(-)}&
&\ar@<0.5ex>[ll]^{\simp(-)} \wide A,\\
&&\\
\fLsbrick A\ar@<0.5ex>[rr]^{\Filt(-)}&
&\ar@<0.5ex>[ll]^{\simp(-)} \fLwide A,\\
&&\\
\brick A\ar@<0.5ex>[rr]^{\Filt(-)}&
&\ar@<0.5ex>[ll]^{\simp(-)} \swide A,}$$ 
where $\simp(-)$ takes the simple objects.
\end{theorem}

\begin{defi}\label{g-finite}
Let $A$ be a finite dimensional algebra.
We call $A$  \emph{$g$-finite} (or \emph{$\tau$-tilting finite})
if $|\twosilt A|<\infty$ (or equivalently, 
$|\ftors A|=|\fLsbrick A|<\infty$).
\end{defi}

In this paper, we always consider $g$-finite algebras. 
The following theorem gives a characterization of  $g$-finiteness.

\begin{thm}\label{finite brick}\cite{DIJ,A1}
Let $A$ be a finite dimensional algebra. Then $A$ is $g$-finite if and only if $\tors A=\ftors A$ if and only if $\sbrick A=\fLsbrick A.$
\end{thm}

Moreover, we recall the following result, which is fundamental in this paper.

\begin{thm}\label{lattice}\cite{DIRRT}
Let $A$ be a finite dimensional algebra and assume that $A$ is $g$-finite. Then $\tors A=\ftors A$ is lattice and semidistributive.
\end{thm}

\subsection{g-fans}
 
Recall that, for an exact category (resp. a triangulated category) $\CC$, the Grothendieck group 
$K_0(\CC)$ is the quotient group of the free abelian group on the set of isomorphism classes $[X]$ of $\CC$
by the relations $[X]-[Y]+[Z]=0$ for all  short exact sequences $0 \to X \to Y \to Z \to 0$
(resp. all triangles $X \to Y \to Z \to X[1]$) in $\CC$. 

Let $A$ be a finite dimensional algebra and assume that $A$ is $g$-finite. 
Let $P_1,\ldots,P_n$ (resp. $S_1,\ldots,S_n$) be all the non-isomorphic indecomposable projective $A$-modules (resp. all the non-isomorphic simple $A$-modules). 
Then the $K_0(\proj A)$ (resp. $K_0(\mod A)$) 
is a free abelian group of rank $n$ having a $\mathbb{Z}$-basis $[P_1],\ldots,[P_n]$ (resp. $[S_1],\ldots,[S_n]$) and it is isomorphic
to $K_0(\Kb(\proj A))$ (resp. $K_0(\Db(\mod A)$). 
Let 
$K_0(\proj A)_\R:=K_0(\proj A)\otimes_{\mathbb{Z}}\R\cong \R^{|A|}.$

\begin{defi}\label{g-fan}
For $T=T_1\oplus\cdots\oplus T_j\in\twopsilt A$ with indecomposable $T_i$, we define a convex polyhedral cone 
\begin{eqnarray*}
C(T) := \{\sum_{i=1}^ja_i[T_i]\mid a_1,\ldots,a_j\ge 0\}\subset K_0(\proj A)_\R.
\end{eqnarray*}
We call the set
\[\Sigma(A):=\{C(T)\mid T\in\twopsilt A\}\]
of cones the \emph{$g$-fan} of $A$.
Moreover, we define 
$$\C(A) := \{C(T)\ |\ T\in\twosilt A\}$$
and the elements of $\C(A)$ is called \emph{chambers} of $\Sigma(A)$.
\end{defi}

Via bijection $\twosilt A\to \C(A)$, $T\mapsto C(T)$, we regard $\C(A)$ as a poset. 
Thus, we have poset isomorphisms $\twosilt A$, $\tors A$ and $\C(A)$ and hence $\Hasse(\twosilt A)\cong\Hasse(\tors A)\cong\Hasse(\C(A))$. 
To specify this poset structure, without particular reference to 
which ground set is used to define it, 
we simply write $\Hasse A:=\Hasse(\twosilt A)$ and  $\jirr(A):=\jirr(\twosilt A)$.

We recall the basic properties of $\Sigma(A).$


\begin{prop}\cite{DIJ}\label{property of g-cone}
Let $A$ be a finite dimensional algebra with $n:=|A|$.
\begin{itemize}
\item[(1)] $\Sigma(A)$ is a nonsingular fan in $K_0(\proj A)_\R$.
\item[(2)] Any cone in $\Sigma(A)$ is a face of a cone of dimension $n$.
\item[(3)] Any cone in $\Sigma(A)$ of dimension $n-1$ is a face of precisely the two cones of dimension $n$.
\item[(4)] $A$ is $g$-finite if and only if $\Sigma(A)$ is complete.
\item[(5)]  $\Sigma(A)$ is sign-coherent, that is, for each $\sigma\in\Sigma(A)$, 
there exists $\epsilon_1,\ldots,\epsilon_n\in\{1,-1\}$ such that $\sigma\subseteq \{\sum_{i=1}^n\epsilon_ia_i[P_i]\mid a_1,\ldots,a_n\ge 0\}$.
\end{itemize}
\end{prop}

For $1\leq i\leq n$, 
let $H_i$ be the hyperplane in $K_0(\proj A)_\R$ which is orthogonal to $[P_i]$ and we call it the \emph{basic hyperplane}. By Proposition \ref{property of g-cone}, if $A$ is $g$-finite, then any  basic hyperplane is contained in the boundaries of the cones of $\Sigma(A)$.

For $L\in\twopsilt^{n-1}A$, we define 
$$H(L) :=\span(C(L)),$$
that is, $H(L)$ is the hyperplane containing $C(L)$. 
We remark that a $g$-fan is not necessary a hyperplane arrangement, 
so that $H(L)$ ($L\in\twopsilt^{n-1}A$) is not necessary contained in the boundaries of the cones of $\Sigma(A)$. 
For example, 
for the quivier $Q:=(1\to 2)$, let $A:=kQ$ be the path algebra of $Q$. Then  
the $g$-fan of $A$ is given as follows:
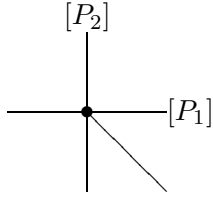
\[{\begin{xy}
0;<3pt,0pt>:<0pt,3pt>::
(5,10)="0",
(-5,10)="1",
(0,0)*{\bullet},
(13,0)*{[P_1]},
(0,12)*{[P_2]},
(0,0)="2",
(10,0)="3",
(-10,0)="4",
(0,10)="5",
(0,-10)="6",
(10,-10)="7",
\ar@{-}"2";"3",
\ar@{-}"3";"4",
\ar@{-}"5";"6",
\ar@{-}"2";"7",
\end{xy}}\]

The aim of this paper is to extend the theory of shards, which was originally introduced by Reading to hyperplane arrangements \cite{Re3,Re4}, to $g$-fans and explain a relationship with several important objects of representation theory.

We also recall the following result. 

\begin{prop}\cite{DIJ}
 Let $L\in\twopsilt^{n-1}A$.
Then the hyperplane $H(L)$ divides $K_0(\proj A)_\R$ into a half-space containing $C(A)$ and a half-space containing $C(A[1])$.
\end{prop}
We call the half-space containing $C(A)$ (resp. $C(A[1])$) the \emph{positive} half-space (resp. \emph{negative} half-space) defined by the hyperplane $H(L)$. Then we have the following important theorem.

\begin{thm}\cite{DIJ}\label{poset hyperplane} 
Let $M,N\in\twosilt A$ such that they are mutation of each other. The following are equivalent, 
where $L\in\twopsilt^{n-1}A$ such that $\add L=\add M\cap\add N$,  

\begin{itemize}
\item[(a)] $M\covers N$
\item[(b)] $C(M)$ belongs to the positive half-space defined by $H(L)$. 
\item[(c)] $C(N)$ belongs to the negative half-space defined by $H(L)$.
\end{itemize}
\end{thm}

\subsection{Reduction technique}

Let $U\in\twopsilt A$. Define $$\twosilt_U A:=\{T\in \twosilt A\ | \ U\in\add T\}.$$

For $\UU:=C(U)$, it is naturally identified with 
$$\C(A,\UU) :=\{\TT\in\C(A)  \mid \UU\textnormal{ is a face of }\TT\}=
\{C(T)\ |\ T\in\twosilt_U A \}.$$

For a fan $\Sigma$ in $K_0(\proj A)_\R\cong\mathbb{R}^n$ and $\sigma\in\Sigma$, we define the reduction of $\Sigma$ at $\sigma$ by
$$\Sigma/\sigma:=\{\pi(\tau)\ |\ \tau\in\Sigma,\ \sigma\subseteq\tau\},$$
where $\pi:\R^n\to\R^n/\R\sigma$ is a natural projection. Then $\Sigma/\sigma$ is a fan in $\R^n/\R\sigma$.

\begin{thm}\label{reduction}\cite{J,A2,AHIKM} 
Let $U\in\twopsilt A$ such that $|U|=i$ $(1\leq i\leq n)$.
\begin{itemize}
\item[(1)] There exists a finite dimensional algebra $B$ such that $|B|=n-i$ and a poset isomorphism 
$$\twosilt_U A\cong \twosilt B.$$

\item[(2)] The isomorphism $K_0(\proj A)_\R/K_0(\add U)_\R\cong K_0(\proj B)_\R$ gives an isomorphism of fans 
\[\Sigma(A)/C(U)\cong\Sigma(B).\]
\end{itemize}
\end{thm}


As a consequence of Theorem \ref{reduction}, we have the following result. 

\begin{cor}\label{cor rank2}
Let $U\in\twopsilt^{n-2} A$.
There exists a finite dimensional algebra $B$ such that $|B|=2$ and  
$\twosilt_U A\cong \twosilt B.$ In particular, there exist 
$$T_{\max},T_{\min},T_1,T_2,\ldots,T_\ell,T_1',T_2',\ldots,T_m'\in\twosilt_U A$$ and $\Hasse(\twosilt_U A)$ is given as follows.
$$\xymatrix{&T_\ell \ar[r]&T_{\ell-1}\ar[r]&\cdots\ar[r]&T_2\ar[r]&T_1\ar[dr]&\\
T_{\max}\ar[ru]\ar[rd]&&&&&&T_{\min}\\
&T_m'\ar[r]&T_{m-1}'\ar[r]&\cdots\ar[r]&T_2'\ar[r]&T_1'\ar[ur]&
}$$

Furthermore, $\Sigma(A)/C(U)\cong\Sigma(B)$ has the following picture:
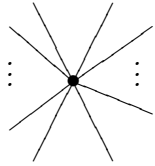
\[{\begin{xy}
0;<3pt,0pt>:<0pt,3pt>::
(5,10)="0",
(-5,10)="1",
(0,0)*{\bullet},
(0,0)="2",
(5,-10)="3",
(-5,-10)="4",
(10,7)="5",
(10,-4)="a",
(-10,5)="6",
(-8,-6)="b",
(-8,7)="7",
(8,2)*{\vdots},
(-8,2)*{\vdots},
\ar@{-}"2";"0",
\ar@{-}"2";"1",
\ar@{-}"2";"3",
\ar@{-}"2";"4",
\ar@{-}"2";"a",
\ar@{-}"2";"5",
\ar@{-}"2";"b",
\ar@{-}"2";"7",
\end{xy}}\]
where the top chamber is the maximum element and the bottom chamber is the minimum element, and these chambers are given by the same two hyperplanes. 
\end{cor}

Let $U\in\twopsilt A$. Via the identification of $\twosilt A$ with $\tors A$, let $\T(U)_{\max}$ (resp. $\T(U)_{\min}$) be the maximum torsion class (resp. minimum torsion class) of $\twosilt_U A$. 
Then we have the following result. 

\begin{prop}\cite[Proposition 2.9]{AIR}\label{max in n-2}
Let $U\in\twopsilt A$. 
Then we have 
$$\T(U)_{\min}=\Fac(H^0(U))\ \ \textnormal{and}\ \ \T(U)_{\max}={}^{\perp}H^{-1}(\nu U).$$
\end{prop}


\section{Shards}
In the rest of this paper, let $A$ be a finite dimensional algebra which is $g$-finite (hence $\Sigma(A)$ is complete).
In this section, we define shards for $g$-fan $\Sigma(A)$ and study their fundamental properties. 
In the case of a hyperplane arrangement, the partial order of chambers is defined by the the set of hyperplanes which separates 
a chamber from the base chamber. 
However, we can not apply this argument to the $g$-fan (since it is not necessary a hyperplane arrangement) 
and hence many previous results in \cite{Re3,Re4} need to be shown in our setting. 
In particular, we will give a bijection between shards of $\Sigma(A)$ and join-irreducible elements of $\C(A)$. Moreover, we will provide a canonical join representation in terms of shards.

First we prepare some terminology.

\begin{defi}
We call an $(n-1)$-dimensional cone $\UU$ of $\Sigma(A)$ a \emph{small wall} if there exists $U\in\twopsilt^{n-1}A$ such that 
$\UU=C(U)$. 
Note that $\UU$ is a face of precisely the two cones of dimension $n$ by Proposition \ref{property of g-cone}.
\end{defi}

\begin{defi}
\begin{itemize}
\item[(1)] 
Take a small wall $\UU\in\Sigma(A)$ and let $H:=\span(\UU)$. 
The \emph{plate} $\PP$ containing $\UU$ is the maximal collection of small walls of $\Sigma(A)$ containing 
$\UU$ and contained in $H$ such that, for any two small walls $\FF,\FF'$ contained in $\PP$, there is a sequence of small walls 
$\FF=\FF_1,\ldots,\FF_k=\FF'$ with $\FF_i$ and $\FF_{i+1}$ intersecting in a cone of dimension $n-2$ for all $1\leq i\leq k-1$. 
We denote by $\mathbb{P}(A)$ the set of all plates of $\Sigma(A)$. 
\item[(2)] Let $\PP\in \mathbb{P}(A)$ be a plate. Let 
$\UU\in\C(A)$ be a small wall such that $\LL:=\UU\cap \PP$ is an $(n-2)$-dimensional cone. 
We say that $\UU$ cuts $\PP$ along $\LL$ if $\UU$ is a facet of maximum or minimum cones of $\C(A,\LL)$ and  
$\PP$ contains neither a facet of maximum nor minimum cones of $\C(A,\LL)$.

\item[(3)] 
For a plate $\PP\in \mathbb{P}(A)$, we remove all points that cut by small walls of $\Sigma(A)$, that is, consider $\PP\backslash\bigcup(\UU\cap\PP)$, where the union is taken over all small walls that cut $\PP$  
(Thus we could divide a plate into a small pieces of connected components). 
The closure of a connected component of $\PP\backslash\bigcup(\UU\cap\PP)$ is called a \emph{shard}. 
Thus, a shard consists of a union of small walls of $\Sigma(A)$ belonging to the same hyperplane. 
We denote by $\shard A$ the set of shards of $\Sigma(A)$.
\end{itemize}

\end{defi}

\begin{exam}\label{exam1}
\begin{itemize}
\item[(1)] Let $A$ be a rank 2 $g$-finite algebra. 
Then $\Sigma(A)$ is given as follows:  
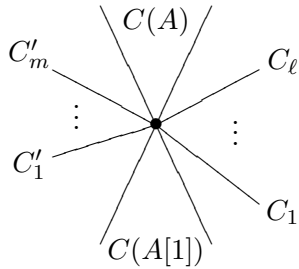
\[{\begin{xy}
0;<3pt,0pt>:<0pt,3pt>::
(7,15)="0",
(-7,15)="1",
(0,0)*{\bullet},
(0,0)="2",
(7,-15)="3",
(-7,-15)="4",
(13,7)="5",
(13,-10)="a",
(-13,5)="6",
(-13,-4)="b",
(-13,7)="7",
(10,0)*{\vdots},
(-10,2)*{\vdots},
(0,13)*{C(A)},
(0,-16)*{C(A[1])},
(16,-11)*{C_1},
(16,8)*{C_{\ell}},
(-16,-5)*{C_1'},
(-16,9)*{C_m'},
\ar@{-}"2";"0",
\ar@{-}"2";"1",
\ar@{-}"2";"3",
\ar@{-}"2";"4",
\ar@{-}"2";"a",
\ar@{-}"2";"5",
\ar@{-}"2";"b",
\ar@{-}"2";"7",
\end{xy}}\]
where $C_1,\cdots,C_\ell,C_1'\cdots,C_m'$ are one-dimensional cones generated by an indecomposable presilting complex. 
Then the shards are illustrated as follows:

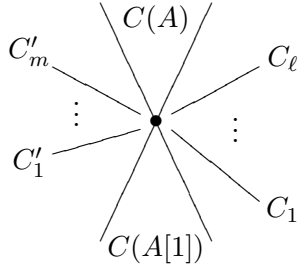
\[{\begin{xy}
0;<3pt,0pt>:<0pt,3pt>::
(7,15)="0",
(-7,15)="1",
(0,0)*{\bullet},
(0,0)="2",
(7,-15)="3",
(-7,-15)="4",
(13,7)="5",
(13,-10)="a",
(-13,5)="6",
(-13,-4)="b",
(-13,7)="7",
(10,0)*{\vdots},
(-10,2)*{\vdots},
(0,13)*{C(A)},
(0,-16)*{C(A[1])},
(16,-11)*{C_1},
(16,8)*{C_{\ell}},
(-16,-5)*{C_1'},
(-16,9)*{C_m'},
\ar@{-}"2";"0",
\ar@{-}"2";"1",
\ar@{-}"2";"3",
\ar@{-}"2";"4",
\ar@{-}(2,-1);"a",
\ar@{-}(2,1);"5",
\ar@{-}(-2,-1);"b",
\ar@{-}(-2,1);"7",
\end{xy}}\]
Namely, two lines intersecting at the origin bounding  $C(A)$ (and $C(A[1])$) are two shards, and $C_1,\cdots,C_\ell,C_1'\cdots,C_m'$ are distinct shards. Note that all shards contain the origin, but they do not continue through the origin. This is a natural analog of the original definition of shards \cite{Re3,Re4}. 
\item[(2)] Let $A$ be a preprojective algebra of Dynkin type. 
Then $\Sigma(A)$ is the Coxeter fan \cite[Theorem 8.4]{AHIKM} and hence it is a hyperplane arrangement. 
Thus, the shards of $\Sigma(A)$ are exactly the same as the original shards of \cite{Re3,Re4}. 
\item[(3)] Let $A$ be a finite dimensional algebra. Then a basic hyperplane is a shard. 
Indeed, a basic hyperplane is a plate by Proposition \ref{property of g-cone} (5). Moreover, for a basic hyperplane, consider 
a small wall $\UU$ in it and 
an $(n-2)$-dimensional cone $\LL=C(L)$ for some $\LL\in\twosilt^{n-2}A$ such that $\LL$ is a face of $\UU$. Then Theorem \ref{poset hyperplane} implies that $\UU$ is a 
facet of the maximum chamber of $\C(A,\LL)$ and hence it is not cut along $\LL$. Thus a basic hyperplane itself is a shard. 
\end{itemize}
\end{exam}

\begin{defi}
\begin{itemize}
\item[(1)] The unique hyperplane containing a shard $\Sigma$ is denoted by $H(\Sigma)$. 
\item[(2)] We call a chamber 
$\TT\in\C(A)$ an
\emph{upper chamber} of shard $\Sigma$ if  $\dim(\TT\cap\Sigma)=\dim(\Sigma)$ and $\TT$ belongs to the positive half-space defined by $H(\Sigma)$.
Or equivalently, an upper chamber of a shard $\Sigma$  has a facet contained in $\Sigma$ such that the chamber adjacent through that facet is lower in the poset $\C(A)$. 
Dually, we can define a \emph{lower chamber} of a shard $\Sigma$. 
We denote by $\Up(\Sigma)$ (resp. $\Lo(\Sigma)$) the set of upper chambers (resp. lower chambers) of $\Sigma$. 
\item[(3)]We say that a shard $\Sigma$ is a \emph{lower shard} of $\RR$ if $\RR$ is an upper chamber of $\Sigma$. 
\item[(4)]For a shard $\Sigma$, we denote the set of small walls in $\Sigma$ by $\sw(\Sigma)$, that is,  $$\sw(\Sigma)=\{C(T)\ |\ T\in\twopsilt^{n-1}A,\ C(T)\in\Sigma\}.$$
\end{itemize}
\end{defi}

The following lemma is an analog of \cite[Lemma 9-7.7]{Re4}.

\begin{lemm}\label{seq of poset}
Let $\Sigma\in\shard A$.
\begin{itemize}
\item[(1)] Let $\MM$ and $\NN$ be chambers of $\Up(\Sigma)$. 
Then there exists a sequence 
$\FF_0:=\MM\cap\Sigma$, $\FF_1,\cdots,\FF_{\ell-1},\FF_\ell=:\NN\cap\Sigma$ of small walls in $\sw(\Sigma)$ such that 
$\FF_{i-1}$ and $\FF_i$ are adjacent  $(1\leq i\leq \ell-1)$.

\item[(2)] For $\FF_i$ defined in (1), we denote by $\RR_i$ (resp. $\QQ_i$) the corresponding upper chamber in $\Up(\Sigma)$ (resp. lower chamber in $\Lo(\Sigma)$), that is, $\RR_i\cap\Sigma=\FF_i$ 
(resp. $\QQ_i\cap\Sigma=\FF_i$)
and $\RR_i$ (resp. $\QQ_i$)
belongs to the positive (resp. the negative) 
half-space defined by 
$H(\Sigma)$. 
Then we have $\RR_{i-1}<\RR_{i}$ and $\QQ_{i-1}<
\QQ_{i}$, or $\RR_{i-1}>\RR_{i}$ and $\QQ_{i-1}>\QQ_{i}$,. In particular, the partial order $\Up(\Sigma)$ is connected. 
\end{itemize}
\end{lemm}

\begin{proof}
(1) follows from the definition. 
(2) Take an $(n-2)$-dimensional cone $\LL_i:=\FF_{i-1}\cap\FF_{i}\cap\Sigma$. 
Then we have $\QQ_{i-1},\QQ_i,\RR_{i-1},\RR_i\in\C(A,\LL_i)$. 
Moreover, by definition of shards, $\Sigma$ is not cut along $\LL_i$, and $\RR_i\covers\QQ_i$. Since Corollary \ref{cor rank2} implies that  
$\C(A,\LL_i)$ is isomorphic to $\C(B)$, where $B$ is an algebra of rank 2, $\RR_i$ (resp. $\QQ_{i-1}$) is the maximum element (resp. the minimum element) of $\C(A,\LL)$ or 
 $\RR_{i-1}$ (resp. $\QQ_{i}$) is the maximum element (resp. the minimum element) of $\C(A,\LL)$. 
Thus, we get the desired conclusion.
\end{proof}

Let $\Sigma$ be a shard and 
$\RR_1,\RR_2,\RR_{3}\in\Up(\Sigma)$. 
For simplicity, we use the same notation as Lemma \ref{seq of poset}, that is, we let $\FF_k:=\RR_k\cap\Sigma$ for 
$k=1,2,3$. Assume that $\RR_1,\RR_2,\RR_{3}\in\Up(\Sigma)$ are adjacent, that is, $\RR_i$ and $\RR_{i+1}$ share an   $(n-2)$-dimensional cone $\LL_i:=\FF_{i}\cap\FF_{i+1}$ for $i=1,2$. 
Under the above setting, we have the following lemma.

\begin{lemm}\label{existence}
Let $\Sigma$ be a shard and 
$\RR_1,\RR_2,\RR_{3}\in\Up(\Sigma)$ three adjacent upper chambers.
If $\RR_1<\RR_2>\RR_{3}$ and $\RR_1\neq\RR_3$, then 
there exists $\MM\in\Up(\Sigma)$ and increasing chains in $\Up(\Sigma)$ such that $\MM<\RR_1$ and $\MM<\RR_3$.
\end{lemm}

\begin{proof}
By our assumption $\RR_{1}<\RR_{2}>\RR_{3}$,  
$\RR_2$ is the maximal element of $\C(A,\LL_1)$ and $\C(A,\LL_{2})$. In particular, $\RR_2$ covers 3 different chambers since it also covers $\QQ_2$, where $\QQ_2$ the corresponding lower chamber of $\RR_2$.  
 
(i) Assume that $|A|= 3$. 
In this case, $\RR_2$ is the unique maximal chamber $C(A)$. Thus $\Sigma$ is a basic hyperplane, and hence there is a unique chamber $\MM$ in $\Up(\Sigma)$ which covers the minimal element $C(A[1])$. 
By Theorem \ref{poset hyperplane}, 
$\MM$ is a unique minimal element of $\Up(\Sigma)$, and 
hence $\MM<\RR_1$ and $\MM<\RR_3$.

(ii) Assume that $|A|=n$. 
Let $\NN$ be $(n-3)$-dimensional cone given by the intersection $\LL_1$ and $\LL_2$. 
By applying Theorem \ref{reduction}, 
we have an isomorphism of fans 
\[\pi:\Sigma(A)/\NN\cong\Sigma(B),\]
where $B$ is a finite dimensional algebra such that $|B|=3$ and 
$\pi(\RR_2)\in \Up(\pi(\Sigma))$ is the maximal chamber in $\Sigma(B)$. Thus 
$\pi(\Sigma)$ is a basic hyperplane and hence there is a unique minimal element $\MM'$ in $\Up(\pi(\Sigma))$ as (i).

Since $\pi$ is a natural projection, 
for each chamber $\PP'$ in $\Up(\pi(\Sigma))$, there exist a chamber $\PP$ in $\Up(\Sigma)$ such that $\pi(\PP)=\PP'$. 
Hence there is an element $\MM$ in 
$\Up(\Sigma)$  
such that $\pi(\MM)=\MM'$, which is a unique minimal upper chamber of $\Up(\Sigma)$ in $\Sigma(A)/\NN$ and, in particular, we have $\MM<\RR_1$ and $\MM<\RR_3$. 
\end{proof}

The following proposition gives a  connection between shards and join-irreducible elements, which is a natural generalization of \cite[Proposition 3.3]{Re3} in our setting.

\begin{prop}\label{unique minimal}
For any shard $\Sigma$, 
there exists a unique minimal element of $\Up(\Sigma)$ and it is join-irreducible. 
Thus, we can define the map 
$$\shard A \longrightarrow\jirr (\C(A)),\ \ \ \Sigma\mapsto\JJ(\Sigma):=\min\Up(\Sigma).$$ 
\end{prop}

\begin{proof}
In the case of $|A|\leq 2$, the statement is clear. Thus we only consider the case of $|A|\geq3$.

(1) First we show that a minimal element is join-irreducible. This is shown by the same argument of \cite[Proposition 2.2]{Re2}. 
Let $\JJ$ be a minimal element of $\Up(\Sigma)$.
Assume that $\JJ$ covers more than one element. 
We can take $\JJ\covers \JJ_1$ and $\JJ\covers \JJ_2$ such that  $\JJ_1$ is a lower chamber of $\Sigma$, that is, $\JJ\cap\JJ_1$ belongs to $\Sigma$. 
Then, we can take an $(n-2)$-dimensional cone $\LL=\JJ\cap\JJ_1\cap\JJ_2$ and $\JJ$ is the maximum of $\C(A,\LL)$. By Corollary \ref{cor rank2}, 
we have the following picture. 

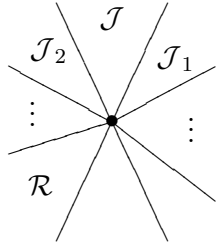
\[{\begin{xy}
0;<3pt,0pt>:<0pt,3pt>::
(7,15)="0",
(-7,15)="1",
(0,0)*{\bullet},
(0,0)="2",
(7,-15)="3",
(-7,-15)="4",
(13,7)="5",
(13,-10)="a",
(-13,5)="6",
(-13,-4)="b",
(-13,7)="7",
(10,0)*{\vdots},
(-10,2)*{\vdots},
(0,13)*{\JJ},
(0,-16)*{},
(16,-11)*{},
(8,8)*{\JJ_1},
(-9,-8)*{\RR},
(-8,9)*{\JJ_2},
\ar@{-}"2";"0",
\ar@{-}"2";"1",
\ar@{-}"2";"3",
\ar@{-}"2";"4",
\ar@{-}"2";"a",
\ar@{-}"2";"5",
\ar@{-}"2";"b",
\ar@{-}"2";"7",
\end{xy}}\]

Therefore, we can take $\RR$, which belongs to $\Up(\Sigma)$ by definition of shard, and $\JJ>\RR$. This contradicts the minimality of $\JJ$. 

(2) Next we show that a minimal element of $\Up(\Sigma)$ is unique. 
Let $\JJ$ be a minimal element and assume that it is not the unique minimum of $\Up(\Sigma)$. 
Let $\Omega$ be the set of chambers of $\Up(\Sigma)$ 
such that they are greater than $\JJ$ and they cover a chamber of $\Up(\Sigma)$ which does not greater than $\JJ$, that is, 
$$\Omega:=\{\mathcal{X}\in\Up(\Sigma)\ |\ \JJ\leq \mathcal{X}\ \textnormal{and}\ \mathcal{Y}\covered \mathcal{X}\  \textnormal{for some }\mathcal{Y}\in\Up(\Sigma)\ \textnormal{s.t.}\ \JJ\nleq \mathcal{Y}\}.$$
Since $\JJ$ is not the unique minimal element, $\Omega$ is not empty from the connectedness of $\Up(\Sigma)$ by Lemma \ref{seq of poset} (2). 
Take a minimal element $\RR$ of $\Omega$ and an increasing chain 
$$\JJ=\RR_0\covered\RR_1\covered\cdots\covered\RR_{k-1}\covered\RR_k=\RR.$$ 
By the definition of $\Omega$, $\RR$ covers some $\mathcal{Y}\in\Up(\Sigma)$ which is not above $\JJ.$ 
Then, by applying Lemma \ref{existence} to $\RR_{k-1}<\RR_k>\mathcal{Y}$, there exists $\MM\in \Up(\Sigma)$ 
such that  
$\RR_{k-1}>\MM$ and $\MM<\mathcal{Y}$. 
Since $\MM<\mathcal{Y}$, we have $\JJ\nleq\MM$. 
On the other hand, consider the decreasing chain from $\RR_{k-1}$ to $\MM$ and we have the following sequence of chambers 

$$\RR_{k-1}=\RR_{k-1}^0\covers\RR_{k-1}^1\covers\cdots\covers\RR_{k-1}^{m-1}\covers\RR_{k-1}^m=\MM.$$ 

Then, since $\JJ\leq\RR_{k-1}$ and $\JJ\nleq\MM$, there exists an integer $1\leq\ell\leq m-1$ such that 
$\JJ\leq \RR_{k-1}^\ell$ and $\JJ\nleq\RR_{k-1}^{\ell+1}$. 
Namely, we have 
$\RR_{k-1}^\ell\in \Omega$, but this contradicts the minimality of $\RR$ since $\RR_{k-1}^\ell<\RR$. Therefore, $\JJ$ is the unique minimal element of $\Up(\Sigma)$.
\end{proof}

Next we recall the following important result. 

\begin{thm}\cite{T2,RST}\label{T2,RST}
For a Hasse arrow $a\to b\in\Hasse A$, 
we have the minimum element in $\{x\mid b\vee x =a\}$ and it is join-irreducible. In particular, 
we have a map 
$$\gamma:\Hasse A\longrightarrow \jirr A,\ \ \ (a\to b) \mapsto \min\{x\mid b\vee x =a\}.$$
\end{thm}

We also define the following terminology. 

\begin{defi}\label{hasse to shard}
Given $\MM\covers\NN$ in $\C(A)$, there is a unique shard containing $\MM\cap\NN$, which we denote by $\Sigma(\MM\covers\NN)$.
If $\MM$ is join-irreducible, then there exists a unique element $\MM^*$ such that $\MM\covers\MM^*$. 
In this case, we simply denote by $\Sigma(\MM)$ instead of $\Sigma(\MM\covers\MM^*)$. Thus, via identification $\C(A)$ with $\twosilt A\cong\tors A$,  
 we can define a map 
$$\Sigma:\Hasse A\longrightarrow \shard A,\ \ \ (a\to b) \mapsto \Sigma(a\covers b).$$
\end{defi}

Theorem \ref{T2,RST} and Definition \ref{hasse to shard} can be related as follows. 

\begin{cor}\label{poset=shard}
Let $a\to b$ be a Hasse arrow of $\Hasse A$.
Then we have 
$$\JJ(\Sigma(a\to b))=\gamma(a\to b).$$
\end{cor}

\begin{proof}
Let $j:=\JJ(\Sigma(a\to b))$. 
It is enough to show that 
$b\vee j=a$ and $b\wedge j=j^*$ since no element smaller than $j$ can be in the set $\{x\mid b\vee x =a\}$ and hence $j$ is the minimum (see \cite[Lemma 2.7]{En} for example).  
Let $\Sigma:=\Sigma(a\to b)$. By Propositon \ref{unique minimal}, 
we can take $\RR_i\in \Up(\Sigma)$ and $\QQ_i\in \Lo(\Sigma)$ ($1\leq i\leq k$) such that 
$j=:\RR_0<\RR_1<\cdots<\RR_{k}:=a$, 
$j^*=:\QQ_0<\QQ_1<\cdots<\QQ_{k}:=b$ and $\RR_{i-1},\RR_i,\QQ_{i-1},\QQ_i\in\C(A,\LL_i)$, where 
$\LL_i:=\RR_{i-1}\cap\RR_i\cap\Sigma=\QQ_{i-1}\cap\QQ_i\cap\Sigma$ is an  $(n-2)$-dimensional cone.

Now observe $\RR_0\nleq\QQ_{k}.$ 
Indeed, assume that $\RR_0\leq \QQ_{k}.$ 
Since we have $\RR_0\vee\QQ_1=\RR_1$ by Proposition \ref{max in n-2}, 
we get $\RR_0\vee Q_1=\RR_1\leq\QQ_{k}$ because $\QQ_1\leq\QQ_{k}$ and $\RR_0\leq \QQ_{k}$. 
Similarly, since $\RR_1\vee Q_2=\RR_2$, we get 
$\RR_2\leq\QQ_{k}$.  
By repeating this argument, we finally get 
$\RR_{k-1}\leq\QQ_{k}$, which is contradiction because Corollary \ref{cor rank2} implies that the poset of  $\C(A,\LL_{k})$ is given as follows :
$$
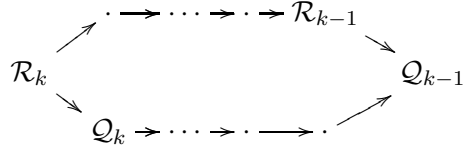
\xymatrix@C=9pt@R=6pt{& \cdot\ar[r]\ar[r]&\cdots\ar[r]&\cdot\ar[r]&\RR_{k-1}\ar[dr]&\\
\RR_{k}\ar[ru]\ar[rd]&&&&&\QQ_{k-1}\\
&\QQ_{k}\ar[r]\ar[r]&\cdots\ar[r]&\ar[r]\cdot&\ar[ur]\cdot&
}$$

Thus, we conclude $\RR_0\nleq\QQ_{k}$ and similarly $\RR_j\nleq\QQ_{k}$ for any $0\leq j\leq k-1$.  
Hence we have $\RR_0\vee\QQ_{k}=\RR_{k}$ and $\RR_0\wedge\QQ_{k}=\QQ_{0}$.
\end{proof}

Then we can establish a bijection between $\jirr (\C(A))$ and $\shard A$. 
Moreover, shards provide a canonical join representation of $\C(A)$, which is an analog of \cite[Theorem 3.6]{Re3}.

\begin{thm}\label{bij jirr and shard}
\begin{itemize}
\item[(1)]
There exists a bijection
$$\xymatrix{\jirr (\C(A))\ar@<0.5ex>[rr]^{\Sigma(-)}&
&\ar@<0.5ex>[ll]^{\JJ(-)} \shard A
}$$

\item[(2)] For any $\RR\in\C(A)$, we have 
$$\RR=\bigvee_{i=1}^\ell\JJ(\Sigma_i),$$ 
where $\Sigma_i$ runs over all lower shards of $\RR$ and it is a canonical join representation. 
\end{itemize}
\end{thm}

\begin{proof}
\begin{itemize}
\item[(1)] By Proposition \ref{unique minimal}, we can show  the map $\JJ(-)$ is well-defined and injective. 
We will show that the map $\JJ(-)$ is surjective. 
Let $\MM\in\jirr(\C(A))$.
Then $\MM\in\Up(\Sigma)$ and the only chamber covered by $\MM$ is in $\Lo(\Sigma(\MM))$. 
Since $\JJ(\Sigma(\MM))\leq\MM$, we conclude $\JJ(\Sigma(\MM))=\MM$.
\item[(2)] By Theorem \ref{lattice}, $\C(A)$ is 
semidistributive and hence the result \cite[Theorem 1]{Go} implies that any element of $\C(A)$ admits a  canonical join representation. 
Moreover, \cite[Lemma 19]{B} implies 
a  canonical join representation of $\RR$ is given by $$\displaystyle\bigvee_{\RR{\,\,>\!\!\!\!\cdot\,\,\,}\RR'}\{\gamma(\RR\to \RR')\},$$ where $\RR'$ runs over all elements covered by $\RR$. 
Then Corollary \ref{poset=shard} implies that 
$\gamma(\RR\to \RR')=\JJ(\Sigma(\RR\to \RR'))$ and hence $$\RR=\bigvee_{\RR{\,\,>\!\!\!\!\cdot\,\,\,}\RR'}\JJ(\Sigma(\RR\covers\RR'))=\bigvee_{i=1}^\ell\JJ(\Sigma_i).$$ 
Consequently, we get the desired result.
\end{itemize}
\end{proof}

\section{Shards and semistable regions}
In this section, we discuss a relationship between shards, bricks, semistable modules and wide subcategories. 
In particular, we show that a shard gives precisely the semistable 
region of a brick and it allows us to give one-to-one correspondence between the set of shards and bricks.

Let $A$ be a finite dimensional algebra which is $g$-finite. 
We define the Euler form 
$$\langle-,-\rangle:K_0(\proj A)\times K_0(\mod A)\to \mathbb{Z}$$
by 
$$\langle T,X \rangle:=\sum_{i\in\mathbb{Z}}(-1)^i\dim_k\Hom_{\Db(\mod A)}(T,X[i]
)$$ 
for any $T\in\Kb(\proj A)$ and $X\in\Db(\mod A)$, which is a non-degenerate $\mathbb{Z}$-bilinear form. 
The Euler form is naturally extended to an $\mathbb{R}$-bilinear form $$\langle -,- \rangle:K_0(\proj A)_{\mathbb{R}}\times K_0(\mod A)_{\mathbb{R}}\to \mathbb{R},$$
where $K_0(\proj A)_{\mathbb{R}}:=K_0(\proj A)\otimes_{\mathbb{Z}}{\mathbb{R}}$ and $K_0(\mod A)_{\mathbb{R}}:=K_0(\mod A)\otimes_{\mathbb{Z}}{\mathbb{R}}$. 
We regard $\theta\in K_0(\proj A)_{\mathbb{R}}$ as an $\mathbb{R}$-linear form 
$\langle \theta,- \rangle:K_0(\mod A)_{\mathbb{R}}\to \mathbb{R}$ and we write $\theta(M):=\langle \theta,M \rangle$.


We recall the notion of semistable modules as follows. 

\begin{defi}\label{semistable wide}
Let $\theta\in K_0(\proj A)_{\mathbb{R}}$.
\begin{itemize}
\item[(1)] We call a module $M\in\mod A$ $\theta$-\emph{semistable} (resp. $\theta$-\emph{stable}) if $\theta(M)=0$ and, for any quotient module $M'$ of $M$, we have $\theta(M')\geq0$ (resp. for any nonzero proper quotient module $M'$ of $M$, we have $\theta(M') > 0$). 
Moreover, we define the $\theta$-semistable subcategory $\W(\theta)$ as the full subcategory consisting of all the $\theta$-semistable modules in $\mod A$. It is a wide subcategory and the simple objects of $\W(\theta)$ are precisely $\theta$-stable modules (see \cite[Proposition 3.24]{BST} for example). 
 
\item[(2)] For $\RR\subset K_0(\proj A)_{\mathbb{R}}$. 
We define 
$$\W(\RR):=\{X\in\mod A\mid X:\theta\textnormal{-semistable for all }\theta\in\RR\}.$$ 
Note that we have  
$$\W(\RR)=\bigcap_{\theta\in\RR} \W(\theta)$$ 
and in particular $\W(\RR)$ is a wide subcategory. 
\item[(3)] For any $X\in\mod A$, we define 
$$\Theta(X):=\{\theta\in K_0(\proj A)_\mathbb{R}\mid X:\textnormal{$\theta$-semistable}\}.$$
Moreover, for a subcategroy $\Cc\subset\mod A$, 
we define 
$$\Theta(\Cc):=\{\theta\in K_0(\proj A)_\mathbb{R}\mid X:\textnormal{$\theta$-semistable for all }X\in\Cc\}.$$
Note that it is a convex cone  and we have 
$$\Theta(\Cc)=\bigcap_{X\in \Cc}\Theta(X).$$
\end{itemize}
\end{defi}

We remark that, for a simple module $S_i$, $\Theta(S_i)$ is nothing but a basic hyperplane $H_i$. 

Then we have the following fundamental property. 
The author would like to thank Haruhisa Enomoto for pointing out this proposition.

\begin{prop}\label{semistable wide2}
\begin{itemize} 
\item[(1)] The map 
$$\xymatrix{\{\textnormal{subcategories of }\mod A\}\ar@<0.5ex>[rr]^(.51){\Theta(-)}&
&\ar@<0.6ex>[ll]^(.48){\W(-)}\{\textnormal{subset of }K_0(\proj A)_{\mathbb{R}}\}}$$ 
gives an antitone Galois connection, that is,  $\Theta$ and $\W$ are order-reversing, and 
$(-)\subset\Theta\W(-)$ and $ (-)\subset\W\Theta(-)$. 
In particular, 
$\Im \Theta$ and $\Im W$ are anti-isomorphic posets.

\item[(2)] In the above setting, we have
$$\Im\W=\wide A.$$

\item[(3)] Let $\W$ be a wide subcategory such that $\{B_i\}_{i=1}^{\ell}$ are 
simple objects of $\W$,i.e, $\W=\Filt(B_1,\ldots,B_\ell)$. 
Then we have 
$$\Theta(\W)=\bigcap_{i=1}^\ell\Theta(B_i).$$

\end{itemize}
\end{prop}

\begin{proof}
(1) This follows from Definition \ref{semistable wide}.

(2) Since $A$ is $\tau$-tilting finite, wide subcategories and semistable categories coincide by  \cite{Y}. Thus this follows from (1).

(3) Clearly, we have $\Theta(\W)\subset\bigcap_{i=1}^\ell\Theta(B_i)$.
On the other hand, let $\theta\in\bigcap_{i=1}^\ell\Theta(B_i)$.
Then, for any $i$, $B_i$ is $\theta$-semistable and belongs to $W(\theta)$. 
Since $W(\theta)$ is wide, we have 
$\W=\Filt(B_1,\ldots,B_\ell)\subset\W(\theta)$ and hence  $\theta\in\Theta(\W).$
\end{proof}

Next we recall the notion of brick labeling.

\begin{definition-theorem}\cite{DIRRT}
There exists an arrow $q:\T\to\U$ of $\Hasse(\tors A)$ if and only if the set of bricks in $\T\cap\U^{\perp}$ contains exactly
one element $B_q$. In this case, 
we have  $\T\cap\U^{\perp}= \Filt(B_q)$. 
We denote $B_q$ by $B(\T\to U)$.
Thus, we have a map 
$$\Hasse A \to\brick A,\ \ \  (a\to  b)\mapsto B(a\to b).$$ 

In this way, for each arrow $q$ of $\Hasse A$, we can associate a brick $B_q$ 
and we call this label \emph{brick labeling}. 
If $a\in\jirr(A)$, then there exists a unique element $a^*$ such that $a\covers a^*$. 
In this case, we write $B(a)$ 
instead of $B(a\to a^*)$. 
\end{definition-theorem}

From now on, we fix $U\in\twopsilt^{n-2}A$. 
Recall that  there exists a finite dimensional  algebra $B$ such that $|B|=2$ and 
$$\twosilt_U A\cong \twosilt B.$$

Thus, there exist 
$$T_{\max},T_{\min},T_1,T_2,\ldots,T_\ell,T_1',T_2',\ldots,T_m'\in\twosilt_U A$$ such that $T_{\min}<T_1<T_2<\cdots<T_\ell<T_{\max}$ and 
$T_{\min}<T_1'<T_2'<\cdots<T_m'<T_{\max}$. 
The Hasse quiver is given as follows:
$$
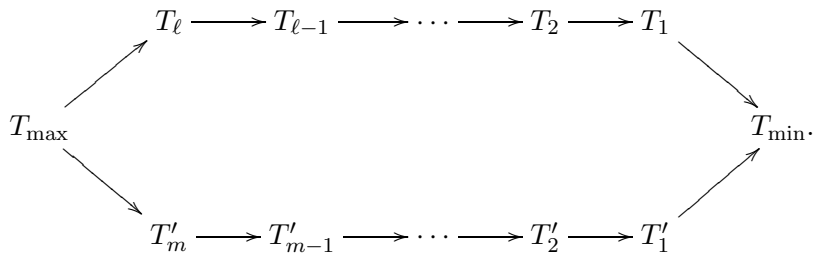
\xymatrix{&T_\ell \ar[r]&T_{\ell-1}\ar[r]&\cdots\ar[r]&T_2\ar[r]&T_1\ar[dr]&\\
T_{\max}\ar[ru]\ar[rd]&&&&&&T_{\min}.\\
&T_m'\ar[r]&T_{m-1}'\ar[r]&\cdots\ar[r]&T_2'\ar[r]&T_1'\ar[ur]&
}$$

By Theorem \ref{poset iso}, we have the corresponding torsion classes, which are defined by  $\T_{\max}:=\Fac(H^0(T_{\max}))$, $\T_{\min}:=\Fac(H^0(T_{\min}))$, $\T_i:=\Fac(H^0(T_{i}))$ ($1\leq i\leq\ell$) and $\T_j':=\Fac(H^0(T_{j}'))$ ($1\leq j\leq m$).
We also define bricks $B_i:=B(\T_{i}\to\T_{i-1})$  ($1\leq i\leq\ell$) and $B_j':=B(\T_{j}'\to\T_{j-1}')$  ($1\leq j\leq m$) and, 
moreover, $B_{\ell+1}:=B(\T_{\max}\to \T_{\ell})$  (resp. $B_{m+1}':=B(\T_{\max}\to \T_{m}')$).



Let $[\T_{\min},\T_{\max}]$ be the interval of torsion classes in $\tors A$ between  $\T_{\min}$ and $\T_{\max}$ . 
Recall that $[\T_{\min},\T_{\max}]\cong\twosilt_U A$
by Theorem \ref{reduction}. 
We also recall the following result by \cite{J,AP}.

\begin{lemm}\cite{J,AP}\label{brick class}
In the above setting,
\begin{itemize}
\item[(1)] we have a lattice isomorphism 
$$\xymatrix{[\T_{\min},\T_{\max}]\ar[r]^(.55){}&
 \tors B.
}$$ 
Moreover, the isomorphism preserves the brick labeling. 

\item[(2)] we have $B_1\cong B_{m+1}'$ and $B_1'\cong B_{\ell+1}$, 
which are the simple objects of $\T_{\max}\cap\T_{\min}^\perp$.
\end{itemize}
\end{lemm}

\begin{proof} 
(1) These statements are \cite{J,AP}. 
(2) By \cite[Theorem 4.2]{AP}, the simple objects of $\T_{\max}\cap\T_{\min}^\perp$ are  $\{B_1,B_1'\}$, which coincide with $\{B_{m+1}',B_{\ell+1}\}$. 
Moreover, since $B_1\in\T_i$ and  $B_1'\notin\T_i$ 
(resp. $B_1'\in\T_j'$ and  $B_1\notin\T_j'$ )
for $1\leq i\leq\ell$ (resp. for $1\leq j\leq m$), 
we get $B_1=B_{m+1}'$ and $B_1'=B_{\ell+1}$. 
\end{proof}

Then we get the following conclusion. 

\begin{lemm}\label{simple difference}
$B_1,B_2,\ldots,B_\ell,B_1',B_2',\ldots,B_m'$ are all distinct.
\end{lemm}

\begin{proof}
By our assumption,  
$\T_1,\T_2,\ldots,\T_\ell,\T_1',\T_2',\ldots,\T_m'$ are all distinct. By Lemma \ref{brick class}, we can check that 
$\T_i$ (resp. $\T_j'$) is the minimal torsion class containg $B_i$ (resp. $B_j'$) . Thus we get the conclusion. 
\end{proof}

We also recall the following fundamental result by \cite{BST,Y}. 

\begin{prop}\cite{BST,Y}\label{Yurikusa}
Let $\TT\covers \TT'$ in $\C(A)$. Then we have 
$$\W(\TT\cap \TT') = \Filt(B(\TT\to\TT')).$$
\end{prop}

Using the above result, we can give the following corollary. 

\begin{cor}\label{semistable region}
\begin{itemize}
\item[(1)] We have $$\W(\TT_{\max}\cap \TT_{\ell}) = \W(\TT_{\min}\cap \TT_{1}')\ \textnormal{and} \ \  \W(\TT\cap \TT'_m) = \W(\TT_{\min}\cap \TT_{1}).$$
\item[(2)] 
For  arbitrary $1\leq i\leq\ell$ and  $1\leq j\leq m$, 
$\W(\TT_i\cap \TT_{i+1})$ and $\W(\TT_j'\cap \TT_{j+1}')$ are all distinct.
\end{itemize}
\end{cor}

\begin{proof}
(1)  We will only show the first statement. The second one is similar. 
By Proposition \ref{Yurikusa}, we have $\W(\TT_{\max}\cap \TT_{\ell}))=\Filt( B_{\ell+1})$. 
Similarly, we have 
$\W(\TT_{\min}\cap\TT_{1}')=\Filt (B_{1}')$. Thus, by Lemma \ref{brick class}, we get the conclusion. 

(2)  This is shown by the same argument of (1) and Lemma \ref{simple difference}. 
\end{proof}

Finally we establish the following theorem.



\begin{thm}\label{main1}
\begin{itemize}
\item[(1)]

A shard gives precisely the semistable region of a brick. 
Moreover, we have a bijection between the set $\brick A$ of 
bricks  and the set $\shard A$ of shards by 
$$\xymatrix{\Theta(-):\brick A\ar[r]^{}&  \shard A.
}$$

\item[(2)]
We have the following commutative diagrams, and all maps are bijections 
\[
\begin{xy}
(-45,0)*{\jirr (\C(A))}="A",
(45,0)*{\shard A}="B",
(-45,-25)*{\brick A}="C",
(45,-25)*{\swide A.}="D",
(-45,0)*{}="E",
(45,0)*{}="F",
(-45,-25)*{}="G",
(45,-25)*{}="H",

\ar@<1ex> "A";"B"^{\Sigma(-)}
\ar@<1ex> "B";"A"^{\JJ(-)}
\ar@<1ex> "A";"C"^{B(-)}
\ar@<1ex> "C";"A"^{\T(-)}
\ar@<1ex> "B";"D"^{\W(-)}
\ar@<1ex> "D";"B"^{\Theta(-)}
\ar@<1ex> "C";"D"^{\Filt(-)}
\ar@<1ex> "D";"C"^{\simp(-)}
\end{xy}
\]
\end{itemize}
\end{thm}

\begin{proof}
(2) Recall that the upper horizontal maps and lower horizontal maps are both  bijections by Theorems \ref{bij jirr and shard} and \ref{bij sbrick wide}, respectively.
Moreover, we have a bijection 
$$\xymatrix{\jirr (\C(A))\ar@<0.5ex>[rr]^{B(-)}&
&\ar@<0.5ex>[ll]^{\T(-)} \brick A
}$$ 
by \cite[Theorem 3.3]{DIRRT},\cite[Theorem 1.5]{BCZ}. 
Thus the left vertical map is also bijection.

Let $\TT\in\jirr (\C(A))$ and $\TT\covers\TT^*$. 
We will show that the diagram is commutative, that is, 
$$\Theta(\Filt (B(\TT))=\Sigma(\TT).$$

We have $\Theta(\Filt (B(\TT)))=\Theta(B(\TT))$ by Proposition \ref{semistable wide2} and hence it is enough to show $\Theta(B(\TT))=\Sigma(\TT).$

Recall that $B(\TT)$ is semistable on $\TT\cap\TT^*$ by 
Proposition \ref{Yurikusa}. 
Consider the facet of $\TT\cap\TT^*$, which is codimension 2, and given as $\UU_1:=C(U_1)$ for some $U_1\in\twopsilt^{n-2}A$. 
If $\TT\cap\TT^*$ is a facet of the maximum or minimum cone on $\C(A,\UU_1)$, by definition of shards, 
then a small wall in $\C(A,\UU_1)$ adjacent to  $\TT\cap\TT^*$, which we denote by $\LL_1$, is contained in $\Sigma(\TT\cap\TT^*)$. 
In this case, by Corollary \ref{semistable region}, 
$B(\TT)$ is also semistable on $\LL_1$. 

Next we take a facet of $\LL_1$ which is not $\UU_1$, and it is given as $\UU_2:=C(U_2)$ for some $U_2\in\twopsilt^{n-2}A$. 
If $\LL_1$ is a facet of the maximum or minimum cone on $\C(A,\UU_2)$, then 
a small wall in $\C(A,\UU_2)$ adjacent to  $\LL_1$, which we denote by $\LL_2$, is contained in $\Sigma(\TT\cap\TT^*)$. 
Similarly, $B(\TT)$ is also semistable on $\LL_2$. 
We can repeat this argument for small walls 
$\LL_1,\cdots,\LL_k$ of $\Sigma(\TT\cap\TT^*)$, where $\LL_i$ and $\LL_{i+1}$ are adjacent. 
Thus we conclude that $B(\TT)$ is semistable on 
$\Sigma(\TT\cap\TT^*)$ and $\Theta(B(\TT))\supset\Sigma(\TT).$ 

Take a small wall $\LL_{k+1}$ on $H(\TT\cap\TT^*)$, which is adjacent to $\LL_k$ and $\LL_{k+1}\notin\Sigma(\TT\cap\TT^*)$. 
Let $\UU:=\LL_k\cap\LL_{k+1}$, which is an $(n-2)$ dimensional cone. 
Then, by the assumption, $\LL_k$ is not a facet of the maximum or minimum cone on $\C(A,\UU)$ and 
Corollary \ref{semistable region} implies that $B(\TT)$ is not semistable on $\LL_{k+1}$. 
Thus, $B(\TT)$ is never semistable on 
the outside of $\Sigma(\TT\cap\TT^*)$ and hence $\Theta(B(\TT))=\Sigma(\TT)$. 

Hence 
the map $\Theta(-):\swide A \to\shard A$ is well-defined and hence, the map is bijection. 
Since $\Theta(\Filt(-))=\Theta(-)$, (1) immediately follows from (2).
\end{proof}

\section{Shard intersections}
In this section we discuss shard intersections and their several properties. 
One of the main results in this section is a poset isomorphism between shard intersections and wide subcategories. 
Moreover, we study several properties of shard intersections. 

\begin{defi}\label{defi shard int}
Let $\shardint A$ be the set of arbitrary intersections of shards. 
We regard the entire space $V:=K_0(\proj A)_{\mathbb{R}}$ is the intersection of the empty set of shards. 
Then $\shardint A$ is a poset by inclusion and it is lattice. 
The unique maximal element is $V$ and the unique minimal element is the intersection of the set of all shards, which is the origin $0$. 
The meet operation is intersection and the join operation $\Gamma_1$ and $\Gamma_2$ is given by the intersection of all shards containing $\Gamma_1$ and $\Gamma_2$.
\end{defi}

We remark in \cite{Re3} the partial order of $\shardint A$ is defined by reverse containment, so that it is opposite of our partial order.

\begin{exam}
Let $A$ be a rank 2 $g$-finite algebra. 
We use the same terminology of Example \ref{exam1}. 
Then the Hasse quiver of $\shardint A$ is given by 
$$
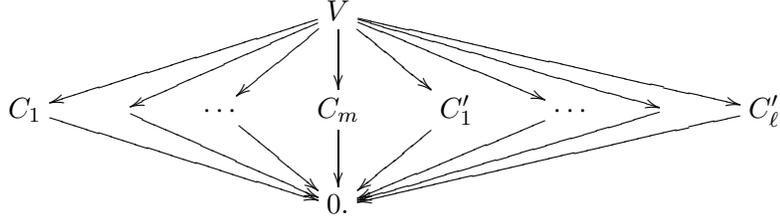
\xymatrix@R=7pt{&&&V\ar[llldd]\ar[lldd]\ar[ldd]\ar[dd]\ar[ddr]\ar[ddrr]\ar[ddrrr]\ar[ddrrrr]&&&&\\
&&&&&&&\\
C_1\ar[rrrdd]&\ar[rrdd]&\cdots\ar[rdd]&C_m\ar[dd]&C_1'\ar[ddl]&\ar[lldd]\cdots&\ar[llldd]&C_\ell'\ar[ddllll]\\
&&&&&&&\\
&&&0.&&&&\\
}$$
\end{exam}

To discuss our main result, we recall the result of 
\cite[Propsition 3.13]{AHIKM} (which was also discussed in \cite{DIRRT}, \cite{BCZ}, \cite{A1}). 

\begin{thm}\label{ahikm}
We have a bijection 
$$\xymatrix{\C (A)\ar@<0.5ex>[rr]^{\{B(-)\}}&
&\ar@<0.5ex>[ll]^{\vee\{\T(-)\}} \sbrick A,
}$$ 
where we define $\{B(\TT)\}:=\{B(\TT\to\TT')\ |\ \TT\covers\TT'\}$ and 
$\vee\{\T(\sS)\}:=\bigvee\{\T(B_i)\}_{i=1}^\ell$ for 
$\sS:=\{B_1,\ldots,B_\ell\}\in\sbrick A$.
\end{thm}

Then we give the following theorem.

\begin{thm}\label{bij jirr and shard2}
\begin{itemize}
\item[(1)] 
There exists an anti-isomorphism 

$$\xymatrix{\wide A\ar@<0.5ex>[rr]^(.45){\Theta(-)}&
&\ar@<0.6ex>[ll]^(.55){\W(-)} \shardint A.
}$$ 
\item[(2)]
We have the following commutative diagrams, and all maps  are 
bijections  
\[
\begin{xy}
(-45,0)*{\C (A)}="A",
(45,0)*{\shardint A}="B",
(-45,-25)*{\sbrick A}="C",
(45,-25)*{\wide A}="D",
(-45,0)*{}="E",
(45,0)*{}="F",
(-45,-25)*{}="G",
(45,-25)*{}="H",

\ar@<1ex> "A";"B"^{\bigcap\Sigma(-)}
\ar@<1ex> "B";"A"^{\bigvee\JJ(-)}
\ar@<1ex> "A";"C"^{\{B(-)\}}
\ar@<1ex> "C";"A"^{\vee \TT(-)}
\ar@<1ex> "B";"D"^{\W(-)}
\ar@<1ex> "D";"B"^{\Theta(-)}
\ar@<1ex> "C";"D"^{\Filt(-)}
\ar@<1ex> "D";"C"^{ \simp(-)}
\end{xy}
\]
where we define $\bigcap\Sigma(\TT):=\bigcap_{\TT{\,\,\cdot\!\! >\,\,}\TT'}\Sigma(\TT\covers\TT')$ and $\bigvee\JJ(\Gamma):=
\bigvee_{\Sigma\supset\Gamma}\JJ(\Sigma)$.

\end{itemize}
\end{thm}

\begin{proof}
(1) By Proposition \ref{semistable wide2}, 
we have an anti-isomorphism $\Im\Theta\cong\Im W=\wide A$.
Hence it is enough to show $\Im \Theta=\shardint A$. 
Let $\Cc \subset\mod A$ be a subcategory. 
Then $\Theta(\Cc)=\Theta (\W (\Theta(\Cc)))$. 
Since $\W (\Theta(\Cc))$ is wide, 
there exists a 
semibrick $\{B_1,\cdots,B_\ell\}$ such that $W (\Theta(\Cc))=\Filt(B_1,\ldots,B_\ell)$.
Moreover, Theorem \ref{main1} implies that there exists a shard $\Sigma_i$ such that $\Sigma_i=\Theta(B_i)$. 
Thus, Proposition \ref{semistable wide2} implies $$\Theta (\W (\Theta(\Cc)))=\Theta (\Filt(B_1,\ldots,B_\ell))=\bigcap_{i=1}^\ell\Theta(B_i)=\bigcap_{i=1}^\ell\Sigma_i$$ and hence $\Im\Theta\subset\shardint A$. 

On the other hand, let  ${\bigcap}_{i=1}^m\Sigma_i\in\shardint A$. 
Then Theorem \ref{main1} implies that there exists a brick $B_i$ such that $\Sigma_i=\Theta(B_i)$. 
Then, for $X:=B_1\oplus\cdots\oplus B_m$, 
we have $\Theta(X)=\bigcap_{i=1}^m\Theta(B_i)=\bigcap_{i=1}^m{\Sigma_i}$ and hence $\Im\Theta\supset\shardint A$.

(2) Left vertical maps, lower horizontal maps and right vertical maps are all bijections by Theorem \ref{ahikm}, Theorem \ref{bij sbrick wide} and (1), respectively. 
Thus, the commutativity of Theorem \ref{main1} implies the commutativity $\bigcap\Sigma(-)=\Theta\circ\Filt\circ\{B(-)\}$ and the map $\bigcap\Sigma(-)$ is bijection. 
We will show that $\bigvee\JJ\circ\bigcap\Sigma(-)$ is the identity map. 

Let $\TT\in\C(A)$ and 
$\TT=\bigvee_{i=1}^\ell\JJ(\Sigma_i)$ be a canonical join representation, 
where $\Sigma_i$ runs over all lower shards of $\TT$.  
Let $\Gamma:=\bigcap\Sigma(\TT)$ and we will show $\TT=\bigvee\JJ(\Gamma)$.  
Since $\Gamma\subset\Sigma_i$, it is enough to show that 
for any shard $\Sigma$ containing $\Gamma$, we have $\JJ(\Sigma)\leq \TT$.

For each $\Sigma_i$, the shard contains a facet of the region $\TT$. 
Thus $\Gamma$ contains the face $F$ of $\TT$ obtained by intersecting the facets of $\TT$ separating $\TT$
from regions $\TT'$ having $\TT\cdot\!\! >\,\TT'$. 
Then any shard $\Sigma$ containing $\Gamma$ contains $F$. Hence there is a region $\PP\in U(\Sigma)$ such that $\PP$ contains $F$. 
By Theorem \ref{reduction}, 
the set $I:=\{\TT\in\C(A)\mid F\subset \TT\}$
is an interval in $\C(A)$, and $\TT$ is the maximal element of $I$. 
Therefore we have $\PP\in I$ and $\JJ(\Sigma)\leq\PP\leq \TT$. 
Consequently $\bigvee\JJ\circ\bigcap\Sigma(-)$ is the identity map. 
\end{proof}

As a corollary, we get the following result, which can be regarded as a refinement of \cite{T1}. Note that it is also shown in \cite[Corollary 4.33]{En} in a different formulation 
(see also Proposition \ref{poset iso wide shard} below).

\begin{cor}
Let $\Delta$ be a Dynkin graph, $\Pi_\Delta$ the preprojective algebra of $\Delta$ and $\Sigma(\Delta)$ the Coxeter fan of $\Delta$. 
Then the shard intersection poset of $\Sigma(\Delta)$ is anti-isomorphic to
$\wide\Pi_\Delta$.
\end{cor}

\begin{proof}
By \cite[Theorem 8.4]{AHIKM}, we have $\Sigma(\Pi_\Delta)$ coincides with $\Sigma(\Delta)$. 
Thus Theorem \ref{bij jirr and shard2} implies the conclusion. 
\end{proof}


We fix a 2-term presilting complex $U$ and 
the corresponding face $\UU:=C(U)$ of $\Sigma(A)$. 

\begin{prop}\label{wide surj}
\begin{itemize}
\item[(1)] Let $\W(U):={}^{\perp}H^{-1}(\nu U)\cap\Fac(H^0(U))^{\perp}$. 
Then $$\W(U)=\W(\UU)$$ and $\sharp\simp(\W(U))=|A|-|U|$. Thus it is a semistable wide subcategory. Moreover, 
there exists a finite dimensional algebra $B$ such that 
equivalence $\W(U)\cong\mod B$. 
\item[(2)] 
There exists surjection $\twopsilt A\to \wide A$, 
$U\mapsto \W(U)$.
\end{itemize}
\end{prop}

\begin{proof}
The statement is the consequence of \cite{BST} and \cite{Y}, where 
they gave an explicit relationship between semistable regions and the associated wide subcategories. 
For the convenience of the reader, we provide a brief explanation. 
\cite[Theorem 3.14]{BST} and \cite[Theorem 1.4]{Y} show that the semistable subcategories given by two-term presilting complexes can be obtained as wide subcategories associated to the two-term presilting complexes. 
It implies $\W(U)=\W(\UU)$ and (2). Moreover the rank of 
the wide subcategories also explicitly given by \cite[Theorem 3.14]{BST}, which completes (1). 
\end{proof}

The following definition provides a canonical description of the shard intersections. 

\begin{defi}
Let $\UU$ be a face of $\Sigma(A)$.
We define 
$$\Gamma(\UU):=\bigcap\{\Sigma(\TT(\UU)_{\max}\covers \TT)|\ \TT(\UU)_{\min}\leq\TT\covered\TT(\UU)_{\max}\},$$
where $\TT(\UU)_{\max}$ (resp. $\TT(\UU)_{\min}$) the maximum chamber (resp. the minimum chamber) of $\C(A,\UU)$. 
Note that it is clearly a shard intersection. 
\end{defi}

\begin{thm}\label{shard int and gamma}
In the above setting, we have 
$$\W(U)=\W(\Gamma(\UU)).$$
\end{thm}

\begin{proof}
Let  $\T(U)_{\max}$ (resp. $\T(U)_{\min}$) be the corresponding torsion class to  $\T(\UU)_{\max}$ (resp. to  $\T(\UU)_{\min}$).  
By Proposition \ref{max in n-2}, 
we have $$[\T(U)_{\min},\T(U)_{\max}]=[\Fac(H^0(U)),{}^\perp H^{-1}(\nu U)].$$
Since  $\W(U)={}^{\perp}H^{-1}(\nu U)\cap\Fac(H^0(U))^{\perp}$, by applying \cite[Theorem 1.4]{AP} to this interval, we have $$\W(U)=\Filt(\{B(\T_{\max}\to\T)|\T(\UU)_{\min}\leq\T\covered\T(\UU)_{\max}\}).$$
By Theorem \ref{bij jirr and shard2}, 
the right-hand side coincides with 
$$\W(\bigcap\{\Sigma(\T_{\max}\to\T))\ |\ \T(\UU)_{\min}\leq\T\covered\T(\UU)_{\max} \})=\W(\Gamma(\UU)).$$ 
Thus we get the conclusion. 
\end{proof}

The following proposition is the same type of Proposition \cite[Proposition 4.4]{Re3}, but our proof is different.

\begin{prop}\label{containg shard}
Let $\UU\in\Sigma(A)$ be a face and $\Gamma$ a shard intersection such that 
$\dim(\UU)=\dim(\Gamma)$ and $\UU\subset\Gamma$. 
Then we have $\Gamma=\Gamma(\UU)$. 
Moreover, $\Gamma$ is the intersection of all shards containing $\UU$.
\end{prop}

For a proof of Proposition \ref{containg shard}, we recall the following basic result. 

\begin{prop}\label{wide and rank}\cite{BM}
Let $A$ be a $\tau$-tilting finite algebra and $\W,\W'$ wide subcategories of $\mod A$. 
If $\W'\subset\W$ and $\sharp\simp\W =\sharp\simp \W'$, then $\W=\W'.$
\end{prop}

\begin{proof}[Proof of Proposition \ref{containg shard}]
Since $\UU\subset\Gamma$, we have $\W(\UU)\supset\W(\Gamma)$. 
Because $\sharp\simp(\W(\UU))=\sharp\simp(\W(\Gamma))$ by $\dim(\UU)=\dim(\Gamma)$, Proposition \ref{wide and rank} implies $\W(\UU)=\W(\Gamma)$. 

On the other hand, Proposition \ref{wide surj} and 
Theorem \ref{shard int and gamma} imply that 
$\W(U)=\W(\UU)=\W(\Gamma(\UU))$. 
Therefore, we get $\W(\Gamma)=\W(\Gamma(\UU))$. 
Since $\Gamma$ and $\Gamma(\UU)$ are both shard intersections, we have $\Gamma=\Gamma(\UU)$ by Theorem 
\ref{bij jirr and shard2}. Thus we get the first assertion. 
Let $\Gamma'$ be the intersection of all shards containing $\UU$. Then we have $\UU\subset\Gamma'\subset\Gamma$ and hence $\UU$ is a full-dimensional face contained in $\Gamma'$. Hence the first assertion implies $\Gamma'=\Gamma(\UU)=\Gamma$.
\end{proof}

Finally we discuss two important properties of shard intersections. 
The first one is a relationship with the \emph{core label order}.

\begin{defi}
\begin{itemize}
\item[(1)] 
For $\RR\in\C(A)$, we define 
$\RR_{\downarrow}:=\bigwedge\{\RR'\in\C(A)|\RR\covers\RR'\}$.
Moreover, for $\RR,\sS\in\C(A)$, we define  
$$\jirr[\RR,\sS]:=\{\gamma(\TT\to \TT'):\RR\leq \TT'\covered \TT\leq \sS\}.$$
Then, for $\RR,\TT\in\C(A)$, we define the \emph{core label order} $\RR\leq_{\textnormal{CLO}}\TT$ by 
$\jirr[\RR_{\downarrow},\RR]\subset\jirr[\TT_{\downarrow},\TT].$

\item[(2)] By Theorem \ref{bij jirr and shard2}, 
any shard intersection $\Gamma\in\shardint A$ is given by $\SSS(\RR)=\Gamma$ for some $\RR\in\C(A)$, where $\SSS:=\cap\Sigma(-):\C(A)\to\shardint A$ is the bijective map. 
Then we define the \emph{shard intersection order}  
$\RR\leq_{\textnormal{SI}}\TT$ if 
$\SSS(\RR)\subset\SSS(\TT).$ Thus the partial order of $\shardint A$ is identified with $(\C(A),\leq_{\textnormal{SI}})$.
\end{itemize}
\end{defi}

The following theorem, which corresponds to  \cite[Proposition 5.7]{Re3}, relates the above two partial orders. 

\begin{prop}\label{poset iso wide shard}
We have 
$$\RR\leq_{\textnormal{SI}}\TT\Longleftrightarrow \RR\geq_{\textnormal{CLO}}\TT.$$
Thus we have an anti-isomorphism $(\C(A),\leq_{\textnormal{SI}})\cong(\C(A),\leq_{\textnormal{CLO}})$.
\end{prop}

A proof of this proposition can be shown by the similar argument of \cite[Proposition 5.7]{Re3}. For the convenience of reader, we give a proof.  

\begin{proof}
The interval $[\RR_{\downarrow},\RR]$ coincides with $\C(A,\UU)$, where $\UU$ is the intersection of $\RR$ with all chambers $\TT$ such that $\TT\covered\RR$. 
Then $\SSS(\RR)$ contains $\UU$ and $\dim(\UU)=\dim(\SSS(\RR))$. 
Thus,  by Proposition \ref{containg shard}, 
a shard $\Sigma$ contain $\SSS(\RR)$ if and only if it contains $\UU$. 
On the other hand, a shard $\Sigma$ contains $\UU$ if and only if it separates two adjacent cones $\VV$ and $\VV'$ of $[\RR_{\downarrow},\RR]$. 
Thus the set $\jirr[\RR_{\downarrow},\RR]$ and 
$\{\JJ(\Sigma)\ |\ \SSS(R)\subset\Sigma\}$ coincide. 

Moreover $\RR\leq_{\textnormal{SI}}\TT$, which is equivalent to $\SSS(\RR)\subset\SSS(\TT)$, if and only if 
$\{\Sigma\ |\ \SSS(\RR)\subset\Sigma\}\supset \{\Sigma\ |\ \SSS(\TT)\subset\Sigma\}$. This is also equivalent to 
$\{\JJ(\Sigma)\ |\ \SSS(\RR)\subset\Sigma\}\supset \{\JJ(\Sigma)\ |\ \SSS(\TT)\subset\Sigma\}$, that is, 
$\jirr[\RR_{\downarrow},\RR]\supset\jirr[\TT_{\downarrow},\TT]$.
\end{proof}

We note that by Theorem \ref{bij jirr and shard2}, $\shardint A$ is anti-isomorphic to $\wide A$. Then, we can also get Proposition \ref{poset iso wide shard} by 
\cite[Theorem 4.25]{En}, which implies that $\wide A\cong(\C(A),\leq_{\textnormal{CLO}})$. 

The second important property is gradedness.
Recall that a finite poset $\PP$ is \emph{graded} if it is equipped with a rank function $\rho:\PP\to \mathbb{N}$ such that 
$\rho(x)=\rho(y)+1$ if $x\covers y$.
Then we have the following result \cite[Proposition 5.1]{Re3}.

\begin{prop}
The partial order $\shardint A$ and hence $\wide A$ is graded. 
More precisely, the rank of $\SSS(\RR)\in\shardint A$ is given by the number of lower hyperplanes of $\RR$ and the rank of $\W\in\wide A$ is given by the number of $\simp(\W)$. 
\end{prop}

\begin{proof}
Assume that $\Gamma\supset\Gamma'$ for $\Gamma,\Gamma'\in\shardint A$. 
Let $\UU'$ be some full-dimensional face in $\Gamma'$.
As in \cite[Lemma 2.10]{Re3}, we can take a face $\UU$ of $\Sigma(A)$ which is full-dimensional in $\Gamma$ and $\UU'$ is a face of $\UU$. 
If $\dim(\Gamma)=\dim(\Gamma')$, then $\dim(\UU)=\dim(\UU')$ and hence $\Gamma=\Gamma(\UU')=\Gamma'$ by Proposition \ref{containg shard}. 
Moreover, if $\dim(\Gamma)>\dim(\Gamma')+1$, then 
$\dim(\UU)>\dim(\UU')+1$. Hence we can take a face $\UU''$ such that $\UU'\subsetneq\UU''\subsetneq\UU$ and 
$\Gamma'\subsetneq\Gamma(\UU'')\subsetneq\Gamma$. 
Thus we get the conclusion for $\shardint A$. The statement for $\wide A$ follows from Theorem \ref{bij jirr and shard2} and the first statement.
\end{proof}

The gradedness of wide subcategories can be directly shown by $\tau$-tilting theory (without using shards). The author would like to thank
Haruhisa Enomoto for pointing out this fact. 
\section*{Acknowledgments} 
The author would like to thank Haruhisa Enomoto for variable discussions and helpful comments. 
He would also like to thank Nathan Reading and Toshiya Yurikusa for answering some questions. He is very grateful to the referee for the careful reading of the paper and for the very kind suggestions that significantly improved the paper.

{}
\end{document}